\documentclass{birkjour}
\usepackage{subfigure}
\usepackage{color}
\usepackage{verbatim}
\usepackage{float}
 \newtheorem{thm}{Theorem}[section]
 \newtheorem{corollary}[thm]{Corollary}
 \newtheorem{lemma}[thm]{Lemma}
 \newtheorem{Proposition}[thm]{Proposition}
 \theoremstyle{definition}
 
 \theoremstyle{remark}
 
 \newtheorem{example}{Example}
 \numberwithin{equation}{section}

 \newcommand{\R}{\mathbb{R}}

    \renewcommand{\H}{\mathcal{H}}
    \newcommand{\M}{\mathcal{M}}
\newcommand{\E}{\mathcal{E}}

\begin{document}

%
%

\title[]
{ \LARGE Envelopes of Bisection Lines of Polygons}

\author[J.Marques]{Joel Albertacci Marques da Silva}

\address{%
Departamento de Matem\'{a}tica- PUC-Rio\br
Rio de Janeiro, RJ, Brasil}
\email{joel.marquesdx@gmail.com}

\author[M.Craizer]{Marcos Craizer}

\address{%
Departamento de Matem\'{a}tica- PUC-Rio\br
Rio de Janeiro, RJ, Brasil}
\email{craizer@puc-rio.br}

\thanks{This paper is part of the Master`s  dissertation of the first author under the guidance of the second author. Both authors want to thank Faperj, CNPq and CAPES (Finance Code 001) for financial support during the preparation of this manuscript. }



\date{June 26, 2024}

\begin{abstract}
A bisection line divides a convex planar curve into two parts with equal areas. It is natural to study the envelope of these lines, which in general present singularities. The polygonal case is particularly inte\-resting, since there are several different notions of a discrete envelope. In this paper, we study three different notions of discrete envelopes of bisection lines and the connections between them.
\end{abstract}

\maketitle

\section{Introduction}

We call {\it bisecting lines} the lines dividing a polygon into two parts with equal area, and denote by $\H$ the envelope of these bisecting lines.
In the case of a triangle, $\H$ is the concatenation of three arcs of hyperbolas (Figure \ref{fig:Triangle}, left, in red). This picture appears in some references (\cite{Ball},\cite{Dunn}). For a regular pentagon, $\H$ is the concatenation of five hyperbolic arcs (Figure \ref{fig:Triangle}, right, in red). For a general polygon, the envelope of bisecting lines is a concatenation of hyperbolic arcs, each of them asymptotic to some pair of sides. In the present paper, we study the structure of these envelopes.

\begin{figure}[htb]
\centering
\subfigure{
\includegraphics[width=.45\textwidth]{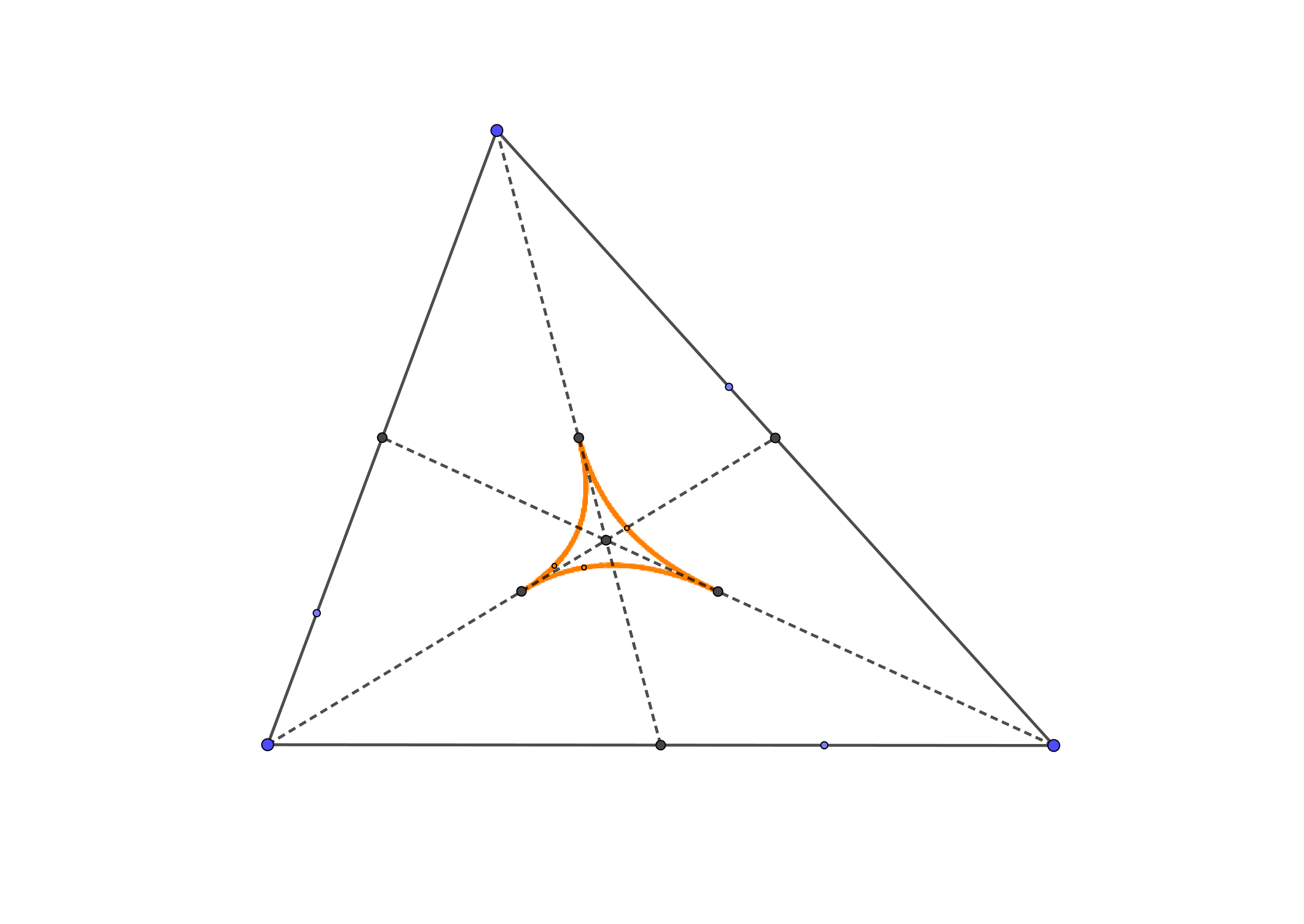}  }
\subfigure{
\includegraphics[width=.45\textwidth]{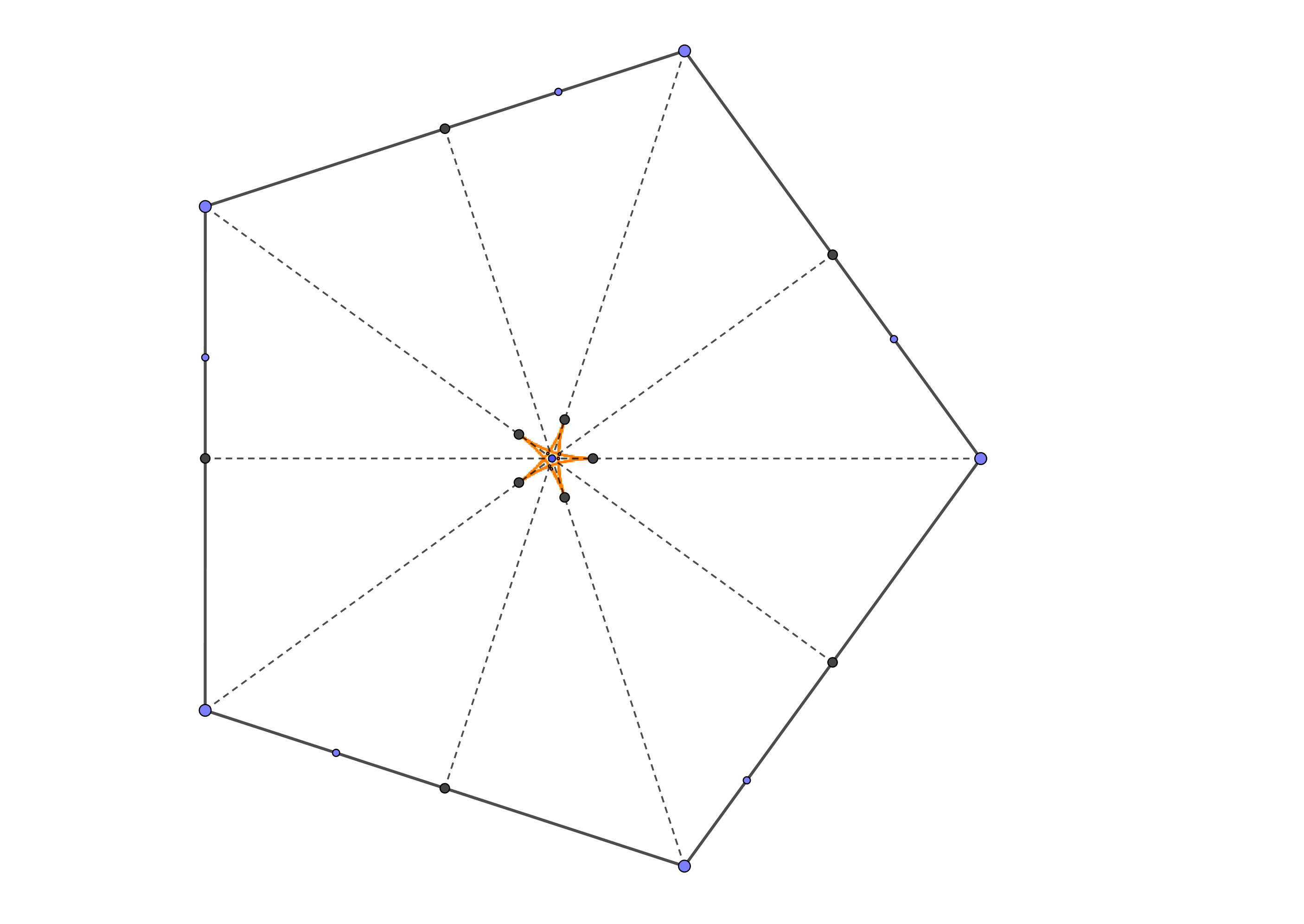} }
\caption{Envelopes of mid-area lines of a triangle and a regular pentagon.}
\label{fig:Triangle}
\end{figure}

Observe in Figure \ref{fig:Triangle}, left, resp. right, that the envelope $\H$ has three, resp. five, special points where the bisecting line cuts it. They are called {\it cusps} and, in the case of a triangle, they coincide with the mid-points of the medians. 
We also observe that each bisecting line through a cup connects a vertex of the polygon with the mid-point of the opposite side, which is not a vertex. However, we can include these points in the set of vertices and it looks natural to do it. Taking this remark into account, we shall consider in this paper only polygons with an even number of vertices, say $2n$, and such that
its {\it principal diagonals}, i.e., diagonals connecting opposite vertices, divide the polygonal region into two equal area parts. We shall call such polygons {\it half-area polygons}.
One can see in Figure \ref{fig:EnvelopeHOctagon} a half-area octagon with four hyperbolic arcs and three cusps.

\begin{figure}[htb]
 \centering
 \includegraphics[width=0.50\linewidth]{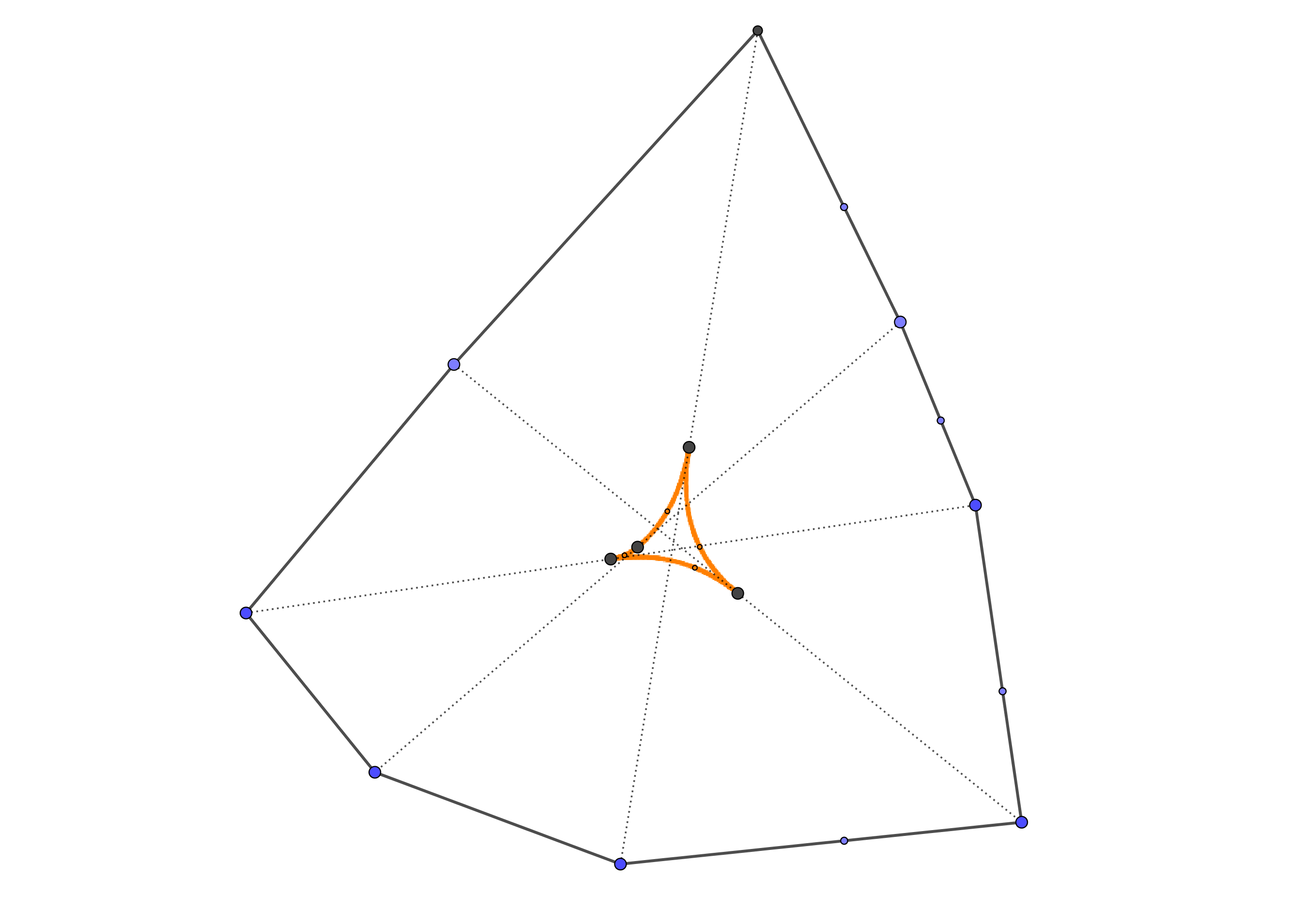}
 \caption{ $\H$ has four hyperbolic arcs but only three cusps. }
\label{fig:EnvelopeHOctagon}
\end{figure}

Besides the envelope of bisecting lines $\H$, there are two other closely related sets in this context. The first one, $\M$, is the $n$-polygon whose vertices are the mid-points of bisecting chords. The second one, $\E$, is the $n$-polygon whose vertices are the intersections of adjacent bisecting lines. We shall see that we can characterize the cusps of $\H$ just by properties of $\M$ and $\E$.  
Based on these properties, we prove that the number of cusps of $\H$ is odd and at least three, which is a discrete version of the well-known smooth counterpart (\cite{Noah}). 

Can we find examples with maximal number of cusps? For $n$ odd, there are interesting examples of $2n$ half-area polygons with $n$ cusps, that we call {\it odd-symmetric} polygons. For such a polygon, the envelope $\E$ reduces to a point, but not $\H$ and $\M$. For $n$ even, we can also construct $2n$ half-area polygons with maximal number, i.e. $(n-1)$, of cusps.  

How about the inverse construction, i.e., can we recover the polygon from $\H$? It is expected that there exists a $1$-parameter family of polygons with $\H$ as envelope of bisecting lines, as occurs with evolutes and involutes, but it is proved in \cite{Noah} that a polygon with no parallel sides is uniquely determined by the envelope $\H$. 
On the other hand, we prove that generically there exists exactly a one-parameter family of polygons with the same $\M$ and $\E$ as the original polygon. Thus, by considering $\M$ and $\E$ in instead of $\H$, we obtain a more natural answer to the recovering problem. 

All the discrete constructions are inspired by their smooth counterparts. The parameterization of smooth curves that inspired the
half-area polygons is is called half-area parameterization. For a given smooth closed convex planar curve, it may be difficult to describe the bisection lines analytically, and so discretization is an alternative. Since the envelopes of bisection lines of half-area polygons preserve the main geometric properties of the envelopes of bisection lines of smooth curves, the present paper proposes, in the spirit of Discrete Differential Geometry, half-area polygons as the correct discretization of smooth curves for the half-area problem.
In \cite{Craizer1} and \cite{Craizer2}, one can find other classes of polygons considered as discrete conterparts of different types of parameterizations of smooth curves.

In Section 2, we study the envelope of bisection lines of a smooth closed convex planar curve. In Section 3, we study the basic facts concerning half-area polygons. In Section 4, we discuss the three types of envelopes of bisection lines of half-area polygons. In Section 5, we discuss the inverse construction for half-area polygons.  
The definitions and results of Sections 3, 4 and 5 should be read comparing with the corresponding definitions and results for smooth curves described in Section 2. We want also to thank Dan Reznik for calling our attention to the envelope of bisecting lines problem.

\section{Smooth Convex Planar Curves}

\subsection{Family of bisection lines}

Consider a parameterization $\gamma:S^1\to\Gamma$ of a smooth planar closed convex curve $\Gamma$. For each $s\in S^1$, denote $l(s)$ a line passing through $\gamma(s)$ 
and that intersect $\Gamma$ at another point, denoted $\gamma(s^*)$. The line $l(s)$ is called a {\it bisection line} if it divides the region enclosed by $\Gamma$ into two regions of equal area. 
It is easy to see that the family of bisection lines exists and is unique. 

\begin{lemma}\label{lemma:BisectionLines}
The family $l(s)$, $s\in S^1$, is the family of bisection lines if and only if, for some fixed $O\in\mathbb{R}^2$, 
\begin{equation}\label{eq:Bisection}
\left[ \gamma'(s), \gamma(s)-O \right]=\left[ \gamma'(s^*), \gamma(s^*) -O\right],
\end{equation}
for any $s\in S^1$.
\end{lemma}

\begin{proof}
In order to keep the formulas shorter, we shall assume that $O$ is the origin of $\mathbb{R}^2$. By Green's theorem, the line $l(s)$ divide $\Gamma$ into two regions of equal areas if and only if 
\begin{equation}\label{eq:BisectionB}
\int_s^{s^*} \left[ \gamma'(t), \gamma(t)\right] dt-\int_s^{s^*} \left[ \gamma'(t^*), \gamma(t^*) \right] dt=0.
\end{equation}
Assume first that $l(s)$ is the family of bisection lines. 
Differentiating Equation \eqref{eq:BisectionB} with respect to $s$ we obtain that
$$
\left[ \gamma'(s^*), \gamma(s^*) \right]\frac{ds^*}{ds}-\left[ \gamma'(s), \gamma(s) \right] =\left[ \gamma'(s), \gamma(s) \right]\frac{ds^*}{ds}-\left[ \gamma'(s^*), \gamma(s^*) \right]
$$
and so
$$
\left[ \gamma'(s^*), \gamma(s^*) \right]=\left[ \gamma'(s), \gamma(s) \right],
$$
thus proving that Equation \eqref{eq:Bisection} holds. Conversely, if Equation \eqref{eq:Bisection} holds, then 
\begin{equation*}\label{eq:Bisection2}
\int_s^{s^*} \left[ \gamma'(t), \gamma(t)\right] dt-\int_s^{s^*} \left[ \gamma'(t^*), \gamma(t^*) \right] dt=C.
\end{equation*}
for some constant $C$. But if we change the r\^oles of $s$ and $s^*$, the left hand of the above equation changes sign, which implies that $C=0$, and so Equation \eqref{eq:BisectionB} holds, which implies that $l(s)$ is a bisection line.
\end{proof}

\subsection{Half-area parameterizations}

From now on $l(s)$ will denote the family of bisection lines of a smooth planar closed convex curve $\Gamma$. 
A parameterization such that
$s^*=s+\pi$ will be called a {\it half-area parameterization}.

\begin{lemma}\label{lemma:ExistenceHAparameter}
Given a smooth planar closed convex curve $\Gamma$, there exist half-area parameterizations $\gamma:S^1\to\Gamma$.
\end{lemma}
\begin{proof}
Consider a smooth parameterization $\gamma_1:S^1\to\Gamma$, where $S^1=\mathbb{R}$ with the identification $z= z+2\pi$.  
Denote $t_0=0^*$ and take any increasing diffeomorphism $h:[0,t_0]\to[0,\pi]$ whose derivatives of all orders coincide at $0$ and $t_0$. This map can be extended to a diffeomorphism $h:S^1\to S^1$ by the condition 
$$
h(t^*)=h(t)+\pi. 
$$
It is then clear that $\gamma=\gamma_1\circ h^{-1}$ is a half-area parameterization of $\Gamma$.
\end{proof}

Denote by 
$[A,B]$ the determinant of the matrix whose columns are $A$ and $B$ and define 
\begin{equation}
a(s)=\left[  \gamma'(s), \gamma(s^*)-\gamma(s) \right].
\end{equation}

\begin{lemma}\label{lemma:HalfArea}
Consider a smooth planar closed convex curve $\Gamma$ and a parameterization $\gamma:S^1\to\Gamma$. Then 
the following conditions are equivalent:

\begin{enumerate}
\item
$\gamma(s)$ is a half-area parameterization.

\item
$$
\frac{ds^*}{ds}=1.
$$

\item
\begin{equation}
a(s)=a(s^*). 
\end{equation}

\end{enumerate}
\end{lemma}

\begin{proof}
The condition $\tfrac{ds^*}{ds}=1$, we conclude that $s^*=s+c$, for some constant $c$. But since $(s^*)^*=s+2\pi$, we conclude that $c=\pi$, thus proving the equivalence between $(1)$ and $(2)$.

Given $s$ and $\Delta s$, let $\Delta s^*=(s+\Delta s)^*-s^*$. 
Denote by $A_1=A_1(s,\Delta s)$ the area of the region bounded by the arc ${\gamma(s)\gamma(s+\Delta s)}$
and the chords $\gamma(s)\gamma(s^*)$ and $\gamma(s+\Delta s)\gamma(s^*)$. Then, by Green's theorem
$$
A_1=\frac{1}{2}\left[ \gamma'(s)\Delta s, \gamma(s^*)-\gamma(s)\right]+O(\Delta s^2).
$$
Now denote by $A_2=A_2(s,\Delta s)$ the area of the region bounded by the arc $\gamma(s^*)\gamma(s^*+\Delta s^*)$
and the chords $\gamma(s^*)\gamma(s+\Delta s)$ and $\gamma(s^*+\Delta s^*)\gamma(s+\Delta s)$. Similarly we have
$$
A_2=\frac{1}{2}\left[ \gamma'(s^*)\Delta s^*, \gamma(s+\Delta s)-\gamma(s^*)\right]+O(\Delta s^2).
$$
On the other hand, by the definition of $s^*$, we have that $A_1=A_2$. Divividing this equality by $\Delta s$ and taking the limit when $\Delta s\to 0$
we obtain
$$
[\gamma'(s),\gamma(s^*)-\gamma(s)]=[\gamma'(s^*),\gamma(s)-\gamma(s^*)]\frac{ds^*}{ds}.
$$
It is now easy to see that conditions $(2)$ and $(3)$ are equivalent.
\end{proof}

\subsection{Envelope of bisecting lines}

From now on, we shall assume that $\gamma:S^1\to\R^2$ is a half-area parameterization of a convex closed planar curve $\Gamma$.

\subsubsection{Basic equations}

Consider the transversal vector field
\begin{equation*}
v(s)=\gamma(s+\pi)-\gamma(s).
\end{equation*}
Then we have that $[\gamma'(s),v(s)]=a(s)$ and $[v(s),v'(s)]=2a(s)$.
Writing
\begin{equation}
v'(s)= \alpha(s) v(s)+\beta(s) \gamma'(s).
\end{equation}
we obtain that 
\begin{equation*}\label{eq:alphabeta}
\alpha(s)=\frac{[\gamma'(s),\gamma'(s+\pi)]}{a(s)},\ \ \beta(s)=-2.
\end{equation*}

\subsubsection{Mid-points of bisecting chords}

For $(x,y)\in\mathbb{R}^2$, let 
\begin{equation*}\label{eq:DefineF}
F(x,y,s)=\left[ (x,y)-\gamma(s),  v(s)  \right].
\end{equation*}
The bisecting lines are defined by $F=0$ and the envelope of the bisecting lines by $F=F_s=0$ (\cite[Ch.5]{Bruce-Giblin},\cite{Nishimura}). Observe that
\begin{equation*}
F_s=\left[ (x,y)-\gamma(s), v'(s) \right]-\left[ \gamma' (s), v(s) \right].
\end{equation*}
Now $F=0$ implies that $(x,y)=\gamma(s)+\lambda(s) v(s)$, for some function $\lambda(s)$. Then $F_s=0$ implies that
$2a\lambda-a=0$, 
which implies that $\lambda=\frac{1}{2}$. 
The envelope of bisecting lines is thus given by 
\begin{equation*}
\E(s)=\gamma(s)+\frac{1}{2} v(s)=\frac{1}{2}\left( \gamma(s)+\gamma(s+\pi) \right).
\end{equation*}
Geometrically $\E$ correspond to the midpoints of bisecting chords.

\subsubsection{Vertices and cusps}

We say that $s$ is a {\it half-area vertex} of $\gamma$ if $\alpha(s)=0$, or equivalently $[\gamma'(s),\gamma'(s+\pi)]=0$. Geometrically, this condition means that the tangents of $\gamma$ at $s$ and $s+\pi$ are parallel. A half-area vertex is said to be {\it ordinary}
if $\alpha'(s)\neq 0$. 
A {\it cusp point} of $\E$ is a point such that $\E'(s)=0$. A cusp point of $\E$ is said to be {\it ordinary} if $\E''(s)\neq 0$.

\begin{Proposition}
Let $\Gamma$ be a smooth convex closed planar curve and $\E$ its bisecting lines envelope. Then
\begin{enumerate}
\item
Cusp points of $\E$ correspond to half-area vertices of $\Gamma$.
\item
Ordinary cusps of $\E$ correspond to ordinary half-area vertices of $\Gamma$.
\end{enumerate}
\end{Proposition}

\begin{proof}
The derivative of $\E$ is 
$$
\E'(s)=\frac{1}{2}\left( \gamma'(s)+\gamma'(s+\pi) \right).
$$
If $\E'(s)=0$, then clearly $\alpha(s)=0$. Conversely, if $\alpha(s)=0$, then \newline 
$v'(s)=-2\gamma'(s)$, which implies that
$$
\E'(s)=\gamma'(s)+\frac{1}{2}v'(s)=0.
$$
We conclude that half-area vertices of $\Gamma$ correspond to cusp points of $\E$.

For the second item, consider a point $s$ with $\alpha(s)=0$. Then
$$
v''(s)=\alpha'(s)v(s)-2\gamma''(s). 
$$
Thus
$$
\E''(s)=\gamma''(s)+\frac{1}{2}v''(s)=\alpha'(s)v(s).
$$
We conclude that ordinary half-area vertices of $\Gamma$ correspond to ordinary cusps of $\E$.
\end{proof}

\subsubsection{Three vertices theorem}

It is easy to check that 
\begin{equation}\label{eq:alphasym}
\alpha(s+\pi)=- \alpha(s).
\end{equation}
This condition implies that the number of half-area vertices in the interval $[0,\pi]$ is odd. Moreover, Equation \eqref{eq:alphasym} implies that
\begin{equation}\label{eq:alphaint}
\int_0^{2\pi} \alpha(s)ds=0.
\end{equation}

\begin{lemma}
We have that
\begin{equation}\label{eq:alphagammaints}
\int_0^{2\pi} \alpha(s)\gamma(s) ds=0.
\end{equation}
\end{lemma}

\begin{proof}
Observe that
\begin{equation*}
\int_0^{2\pi} \alpha(s)\gamma(s) ds=\int_0^{\pi}\alpha(s)\gamma(s)ds+\int_0^{\pi}\alpha(s+\pi)\gamma(s+\pi)ds
\end{equation*}
\begin{equation*}\label{eq:alphagammaint}
=-\int_0^{\pi}\alpha(s)v(s)ds=-\int_0^{\pi}v'(s)+2\gamma'(s) ds=v(0)-v(\pi)+2(\gamma(0)-\gamma(\pi))=0,
\end{equation*}
thus proving the lemma.
\end{proof}

\begin{thm}\label{thm:ThreeVerticesSmooth}
The number of half-area vertices in the interval $[0,\pi]$ is odd and at least $3$.
\end{thm}

\begin{proof}
By Equation \eqref{eq:alphaint}, there exists at least one zero of $\alpha$ in the interval $[0,\pi)$. Assume by contradiction that there is only one such zero $t_0$ in the interval $[0,\pi)$. Write the equation of the bisecting line passing through $\gamma(t_0)=(x_0,y_0)$ as $A(x-x_0)+B(y-y_0)=0$. We may assume that $\alpha>0$ in the region $A(x-x_0)+B(y-y_0)>0$ (or else change the signs of $A$ and $B$). Then 
$$
\int_{t_0}^{t_0+\pi}\alpha (A\gamma_1+B\gamma_2)>0, \ \ \int_{t_0+\pi}^{t_0+2\pi}\alpha (A\gamma_1+B\gamma_2)>0,
$$
where $\gamma=(\gamma_1,\gamma_2)$. But these contradicts Equation \eqref{eq:alphagammaints}
\end{proof}

\subsection{Inverse construction}

Can we recover $\gamma$ from $\E(\gamma)$? In this section, we discuss this question.

\begin{lemma}
Consider a parameterized curve $\bar\gamma:S^1\to\mathbb{R}^2$ and denote by $\bar{l}(s)$ the line connecting
$\bar\gamma(s)$ and $\bar\gamma(s+\pi)$. Then $\bar{l}(s)$ coincides with $l(s)$, for any $s\in S^1$, if and only if 
\begin{equation}\label{eq:GammaC}
\bar\gamma(s)=c(s)\gamma(s)+(1-c(s)) \gamma(s+\pi),
\end{equation}
for some function $c(s)$ with $c(s)=c(s+\pi)$.
\end{lemma}

\begin{proof}
Just observe that $\bar\gamma(s)$ and $\bar\gamma(s+\pi)$ must be in $l(s)$ and that their mean point must coincide with $\E(s)$.
\end{proof}

\begin{lemma}
Consider $\bar\gamma$ given by Equation \eqref{eq:GammaC}, for some function $c(s)$ with $c(s)=c(s+\pi)$.
Then $s^*=s+\pi$, for any $s\in S^1$, if and only if 
\begin{equation}\label{eq:ConstantC}
c'(s)\left[ \gamma(s)-O, \gamma(s+\pi)-O \right]=0,
\end{equation}
for some fixed $O\in\mathbb{R}^2$ and any $s\in S^1$.
\end{lemma}

\begin{proof}
In order to keep the notations shorter, we shall assume that $O$ is the origin of $\mathbb{R}^2$. By Lemma \ref{lemma:BisectionLines}, 
$\bar{l}(s)$ divides the region enclosed by $\bar\gamma$ into equal regions, for any $s\in S^1$ if and only if
$$
\Delta(s)=\left[\bar\gamma(s),\bar\gamma'(s)  \right]-\left[\bar\gamma(s+\pi),\bar\gamma'(s+\pi)  \right]=0,
$$
for any $s\in S^1$. Now straightforward calculations show that
$$
\Delta(s)=(2c-1)\left( [\gamma(s),\gamma'(s)]- [\gamma(s+\pi),\gamma'(s+\pi)]  \right)-2c'(s)[\gamma(s),\gamma(s+\pi)].
$$
But since $l(s)$ divide the region enclosed by $\gamma$ into equal areas, the first parcel of the second member of the above equation vanishes. We conclude that $\Delta(s)=0$ if and only if
$$
c'(s)[\gamma(s),\gamma(s+\pi)]=0,
$$
thus proving the lemma.
\end{proof}

Observe that Equation \eqref{eq:ConstantC} holds if $c(s)$ is constant. On the other side, if $\gamma$ is symmetric with respect to $O$ then Equation \eqref{eq:ConstantC} holds for any $c(s)$. When $\left[ \gamma(s)-O, \gamma(s+\pi)-O \right]=0$, for $s$ in some interval $I\subset S^1$, we say that $\gamma$ is symmetric with respect to $O$ in the interval $I$. We observe that a generic curve $\gamma$ has not intervals of symmetry.

We can now state the main proposition of this section:

\begin{Proposition}\label{prop:RecoveringSmooth}
Assume that $\gamma$ has no intervals of symmetry. Then the envelope of bisecting lines of a curve $\bar\gamma$ coincides with the envelope of bisecting lines of $\gamma$ if and only if $\bar{\gamma}$ is given by Equation \eqref{eq:GammaC}, 
for some constant $c\in\mathbb{R}$. 
\end{Proposition}

\section{Half-Area Polygons}

Consider a polygon $\gamma$ with $2n$ vertices $\gamma(i)$, $1\leq i \leq 2n$. We denote by $\gamma(i+\tfrac{1}{2})$ the side connecting $\gamma(i)$ and $\gamma(i+1)$.
Note that $(i+\tfrac{1}{2})$
here has not a quantitative meaning, it is just a notation to indicate that we are between $i$ and $(i+1)$.  
In the same idea, for a function $f(i)$, we denote the discrete derivative of $f$ by 
$$
f'(i+\tfrac{1}{2})=f(i+1)-f(i).
$$
All indices along the paper are taken modulus $2n$. 

\subsection{Half-area polygons and parallel central diagonals} 

Consider a convex $2n$-polygon $\gamma$. When the principal diagonals of $\gamma$, i.e., the lines $l(i)$ passing through
$\gamma(i)$ and $\gamma(i+n)$, divide the polygon into parts with equal areas, we say that $\gamma$ is a {\it half-area}
polygon.

Note that any convex polygon may be transformed into a half-area polygon by the addition of some vertices. In fact, by including the other intersection of the bisecting lines through the vertices in the list of vertices, we obtain a half-area polygon (see Figure \ref{fig:Triangle}). The new polygon may have collinear vertices, but this is allowed in our class. So, although in this paper we shall consider only half-area polygons, our results can be applied to any convex polygon.

Given a side $\gamma(i+\frac{1}{2})$, we shall call {\it central} the diagonals $\gamma(i)\gamma(i+n+1)$ and $\gamma(i+1)\gamma(i+n)$. Define the vector $v(i)$ by
\begin{equation}\label{definev}
v(i)=\gamma(i+n)-\gamma(i),
\end{equation}
and the quantities $a^{\pm}(i+\tfrac{1}{2})$ by
\begin{equation}\label{defineapm}
a^+(i+\tfrac{1}{2})=[\gamma'(i+\tfrac{1}{2}), v(i)], \ \ \ a^-(i+\tfrac{1}{2})=[\gamma'(i+\tfrac{1}{2}),v(i+1)],
\end{equation}
(see Figure \ref{fig:TrapezoidsAreas}).

\begin{figure}[htb]
 \centering
 \includegraphics[width=0.60\linewidth]{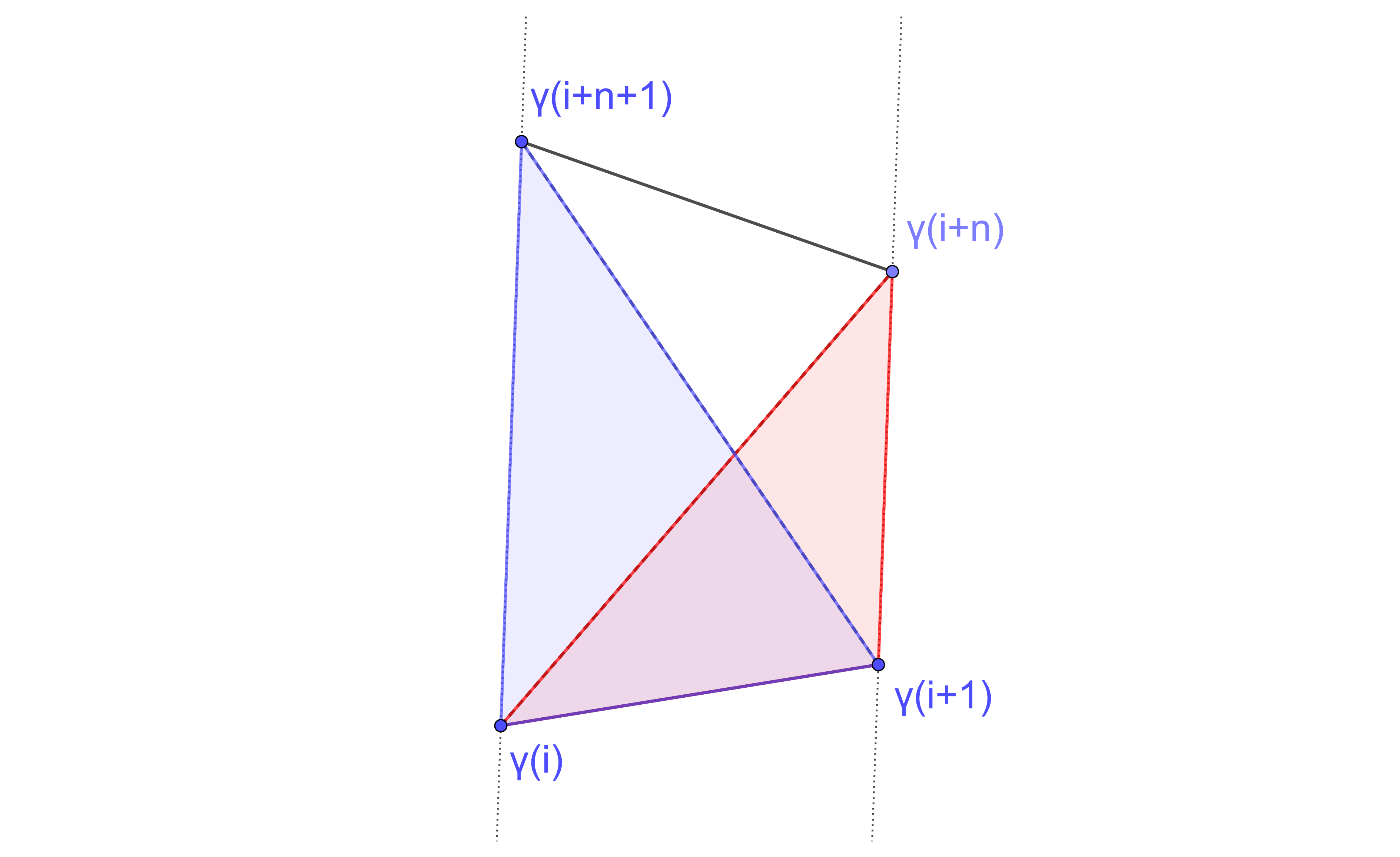}
 \caption{ The parallel central diagonals. The quantities $a^{+}(i+\tfrac{1}{2})$ and $a^{-}(i+\tfrac{1}{2})$ correspond to two times the areas of the triangles $\gamma(i)\gamma(i+1)\gamma(i+n)$ and $\gamma(i)\gamma(i+1)\gamma(i+1+n)$, respectively. }
\label{fig:TrapezoidsAreas}
\end{figure}

\begin{lemma}
Consider a convex $2n$-polygon $\gamma$. The following conditions are equivalent:

\begin{enumerate}
\item
The polygon $\gamma$ is half-area.

\item
Adjacent central diagonals are parallel.

\item
For $1\leq i\leq 2n$,
\begin{equation*}
a^+(i+\tfrac{1}{2})=a^-(i+n+\tfrac{1}{2}).
\end{equation*}

\item
For $1\leq i\leq 2n$,
\begin{equation*}
a^-(i+\tfrac{1}{2})=a^+(i+n+\tfrac{1}{2}).
\end{equation*}

\end{enumerate}

\end{lemma}

\begin{proof}
The equivalence between conditions $(2)$, $(3)$ and $(4)$ is immediate. That $(1)$ implies $(2)$ is also immediate.
Finally, for the implication $(2)\to (1)$, observe that the parallelism between the central diagonals $\gamma(i)\gamma(i+n+1)$ and $\gamma(i+1)\gamma(i+n)$ implies that the area of the polygon $\gamma(i),...,\gamma(i+n),\gamma(i)$ equals the area of the polygon $\gamma(i+1),...,\gamma(i+n+1),\gamma(i+1)$. By induction, we have that the area of the polygon $\gamma(i),...,\gamma(i+n),\gamma(i)$ is independent of $i$. Applying this fact to $\gamma(i+n),\gamma(i+n+1),...,\gamma(i),\gamma(i+n)$, we conclude that the polygon is half-area.
\end{proof}

\subsection{The space of half-area polygons}

\begin{lemma}
The space of $2n$ half-area polygons has dimension $3n$.
\end{lemma}

\begin{proof}
We begin with $n+1$ arbitrary vertices $\gamma(1),...,\gamma(n+1)$, which has dimension $2n+2$. Then $\gamma(n+2)$
belongs to the line parallel to ${\gamma(2)\gamma(n+1)}$ through $\gamma(1)$. By induction,  for $2\leq k\leq n-1$, $\gamma(n+k)$ belongs to the line parallel to ${\gamma(k)\gamma(n+k-1)}$ through $\gamma(k-1)$, which has dimension $n-2$. Finally 
$\gamma(2n)$ belongs to the intersection of the line parallel to ${\gamma(n)\gamma(2n-1)}$ through $\gamma(n-1)$
and the line parallel to ${\gamma(1)\gamma(n)}$ through $\gamma(n+1)$.
\end{proof}

\begin{example}
In case $n=4$, we can choose arbitrarily the points $\gamma(i)$, $1\leq i\leq 5$,  $\gamma(6)$ belongs to the line parallel
to ${\gamma(2)\gamma(5)}$ through $\gamma(1)$ and $\gamma(7)$ belong to the parallel to ${\gamma(3)\gamma(6)}$ through $\gamma(2)$. Finally $\gamma(8)$ belongs to the intersection of the line parallel to ${\gamma(4)\gamma(7)}$ through $\gamma(3)$ and the line parallel to ${\gamma(1)\gamma(4)}$ through $\gamma(5)$(see Figure \ref{fig:HalfAreaOctagon}). The dimension of the space of half-area octagons is $12$.
\end{example}

\begin{figure}[htb]
 \centering
 \includegraphics[width=0.80\linewidth]{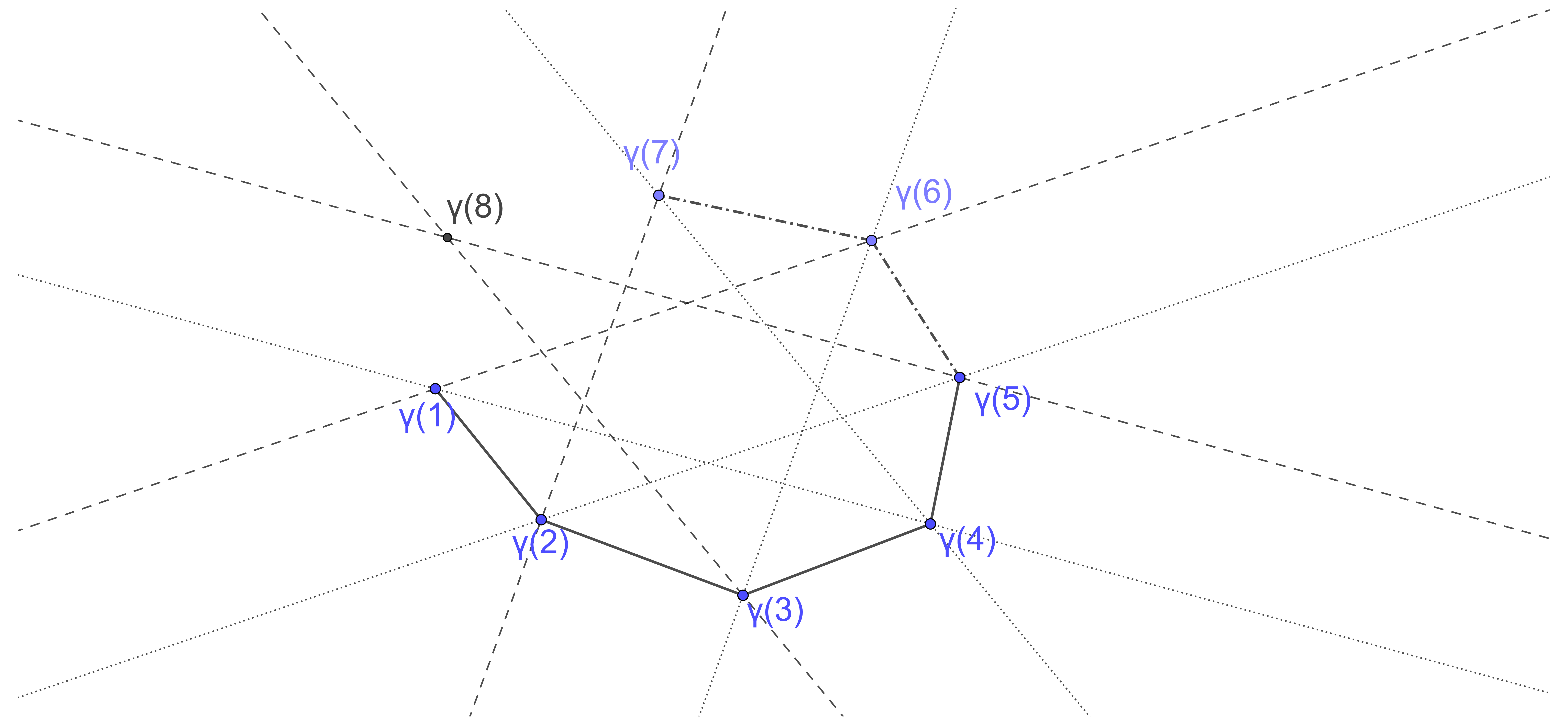}
 \caption{ The construction of a half-area octagon. }
\label{fig:HalfAreaOctagon}
\end{figure}

\subsection{Basic equations}

Define
\begin{equation}\label{definea}
a(i+\tfrac{1}{2})=a^+(i+\tfrac{1}{2})+a^-(i+\tfrac{1}{2}).
\end{equation}
Geometrically, $a(i+\tfrac{1}{2})=a(i+n+\tfrac{1}{2})$ corresponds to the area of the trapezoid wth vertices $\gamma(i)$, $\gamma(i+1)$, $\gamma(i+n)$ and $\gamma(i+n+1)$. Define also
\begin{equation}\label{definedelta}
\delta(i+\tfrac{1}{2})=[\gamma'(i+\tfrac{1}{2}),\gamma'(i+n+\tfrac{1}{2})],
\end{equation}
which measures the degree of non-parallelism of the non-parallel sides of the trapezoid.

\begin{lemma}
For a half-area polygon, the following relations hold:
\begin{enumerate}
\item
$$
[\gamma'(i+\tfrac{1}{2}),v'(i+\tfrac{1}{2})]=\delta(i+\tfrac{1}{2}).
$$

\item
$$
a^-(i+\tfrac{1}{2})-a^+(i+\tfrac{1}{2})=\delta(i+\tfrac{1}{2}).
$$

\item
$$
[v(i),v'(i+\tfrac{1}{2})]=[v(i+1),v'(i+\tfrac{1}{2})]=[v(i),v(i+1)]=a(i+\tfrac{1}{2}).
$$

\end{enumerate}
\end{lemma}

\begin{proof}
For the first item, 
$$
[\gamma'(i+\tfrac{1}{2}),v'(i+\tfrac{1}{2})]=[\gamma'(i+\tfrac{1}{2}),\gamma'(i+n+\tfrac{1}{2})]=\delta(i+\tfrac{1}{2}).
$$

For the second item,
$$
a^+(i+\tfrac{1}{2})+\delta(i+\tfrac{1}{2})=[\gamma'(i+\tfrac{1}{2}),v(i)]+[\gamma'(i+\tfrac{1}{2}),v'(i+\tfrac{1}{2})]=[\gamma'(i+\tfrac{1}{2}),v(i+1)]
$$
$$
=a^-(i+\tfrac{1}{2}).
$$

For the third item, 
$$
[v(i),\gamma'(i+n+\tfrac{1}{2})-\gamma'(i+\tfrac{1}{2})]=[\gamma'(i+\tfrac{1}{2}),v(i)]+[\gamma'(i+n+\tfrac{1}{2}),v(i+n)]=
$$
$$
=a^+(i+\tfrac{1}{2})+a^+(i+n+\tfrac{1}{2}),
$$
thus proving the lemma.
\end{proof}

\begin{lemma}
For a half-area polygon, we can write
$$
v'(i+\tfrac{1}{2})=\alpha_1 v(i)+\beta_1 \gamma'(i+\tfrac{1}{2})=\alpha_2 v(i+1)+\beta_2\gamma'(i+\tfrac{1}{2}),
$$
where
$$
\alpha_1=\frac{\delta(i+\tfrac{1}{2})}{a^+(i+\tfrac{1}{2})}, \ \beta_1=-\frac{a(i+\tfrac{1}{2})}{a^+(i+\tfrac{1}{2})}, \ \alpha_2=\frac{\delta(i+\tfrac{1}{2})}{a^-(i+\tfrac{1}{2})},\ \beta_2=-\frac{a(i+\tfrac{1}{2})}{a^-(i+\tfrac{1}{2})}.
$$
\end{lemma}

\begin{proof}
Observe that
$$
a^+(i+\tfrac{1}{2})\alpha_1=[\gamma'(i+\tfrac{1}{2}),v'(i+\tfrac{1}{2})]=[\gamma'(i+\tfrac{1}{2}),\gamma'(i+n+\tfrac{1}{2})]=\delta(i+\tfrac{1}{2}).
$$
Moreover
$$
a^+(i+\tfrac{1}{2})\beta_1=[v'(i+\tfrac{1}{2}),v(i)]=-a(i+\tfrac{1}{2}).
$$
The proofs for $\alpha_2$ and $\beta_2$ are analogous.
\end{proof}

\section{Envelopes of bisecting lines}

\subsection{ The polygons $\mathcal{E}$, $\mathcal{M}$ and $\mathcal{H}$}

For a half-area polygon $\gamma$, the equation of the bisecting line passing through $\gamma(i)$ is given by $F=0$, where
$$
F(x,y,i)=[(x,y)-\gamma(i),v(i)]=0.
$$
The discrete derivative with respect to the variable $i$ is given by
$$
F'(x,y,i+\tfrac{1}{2})=-[\gamma'(i+\tfrac{1}{2}),v(i)]+[(x,y)-\gamma(i+1),v'(i+\tfrac{1}{2})],
$$
or equivalently,
$$
F'(x,y,i+\tfrac{1}{2})=-[\gamma'(i+\tfrac{1}{2}),v(i+1)]+[(x,y)-\gamma(i),v'(i+\tfrac{1}{2})].
$$

A point $(x,y)$ of the discrete envelope satisfies 
$$
(x,y)=\gamma(i)+\mu v(i)=\gamma(i+1)+\lambda v(i+1)
$$
which implies that
$$
\mu=\frac{a^-(i+\tfrac{1}{2})}{a(i+\tfrac{1}{2})}, \ \ \lambda=\frac{a^+(i+\tfrac{1}{2})}{a(i+\tfrac{1}{2})}.
$$
Thus we write
\begin{equation}\label{eq:DefineE}
{\mathcal E}(i+\tfrac{1}{2})=\gamma(i)+\frac{a^-(i+\tfrac{1}{2})}{a(i+\tfrac{1}{2})} v(i)=\gamma(i+1)+\frac{a^+(i+\tfrac{1}{2})}{a(i+\tfrac{1}{2})} v(i+1).
\end{equation}
Geometrically, $\E(i+\tfrac{1}{2})$ is the intersection of the bisecting lines $l(i)$ and $l(i+1)$ (Figure \ref{fig:Evolutes}, red).
We denote by $\E(i)$ the segment connecting $\E(i-\tfrac{1}{2})$ and $\E(i+\tfrac{1}{2})$. The polygon $\E$ will be called the {\it discrete envelope} of $\gamma$.

In the case of smooth curves, the bisecting lines envelope coincides with the set of mid-points of bisecting chords. For half-area polygons, $\E$ does not coincide with $\M$, the polygon whose vertices  
\begin{equation*}
{\mathcal M}(i)=\frac{1}{2}\left( \gamma(i)+\gamma(i+n) \right)
\end{equation*}
are exactly the mid-points of the principal chords (Figure \ref{fig:Evolutes}, blue). The polygon $\M$ will be called the {\it mid-points envelope} of $\gamma$ and we denote by $\M(i+\tfrac{1}{2})$ the segment connecting $\M(i)$ and $\M(i+1)$. 

The envelope of the bisecting lines of a half-area polygon passing through any point at the sides of the polygon, not just the vertices, will be denoted $\H$. It is a concatenation of hyperbolic arcs ${\mathcal H}(i+\tfrac{1}{2})$ passing through $\M(i)$ and $\M(i+1)$, tangent to
$l(i)$ and $l(i+1)$ at these points, and asymptotic to the support lines of the sides ${\gamma(i+\tfrac{1}{2})}$ and ${\gamma(i+n+\tfrac{1}{2})}$ (Figure \ref{fig:Evolutes}, green). We shall call $\H$ the {\it hyperbolic envelope} of $\gamma$.

\begin{figure}[htb]
 \centering
 \includegraphics[width=0.80\linewidth]{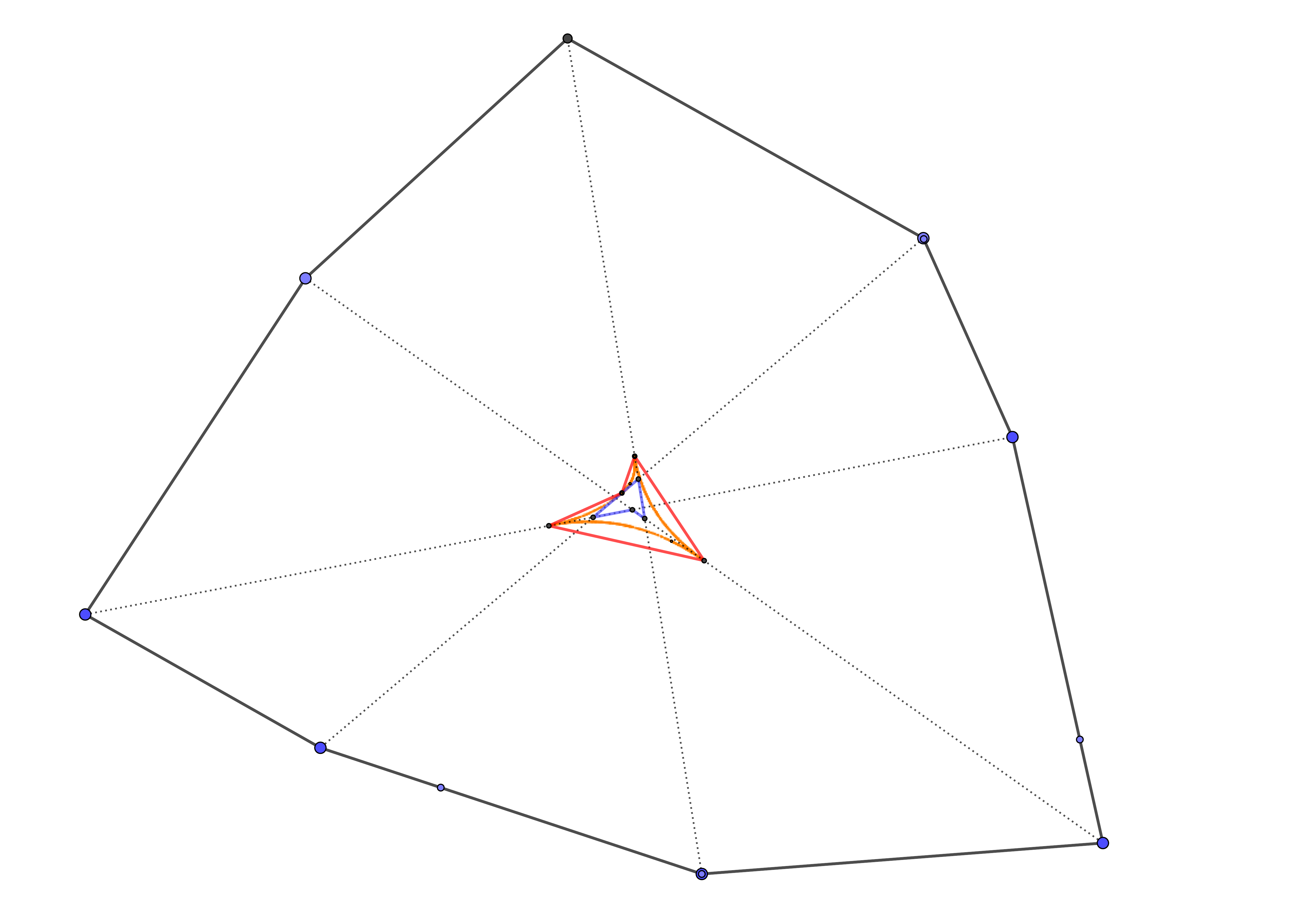}
 \caption{ Discrete envelope $\E$ in blue, mid-points envelope $\M$ in red and the hyperbolic envelope $\H$ in orange. }
\label{fig:Evolutes}
\end{figure}

\subsection{Symmetric and odd-symmetric polygons}

A half-area polygon is {\it symmetric} with respect to a point $O$ if $O$ is the midpoint of the principal diagonals $\gamma(i)\gamma(i+n)$, for every $1\leq i\leq 2n$. It is clear that the symmetry condition is equivalent to $\M=O$ and also to $\H=O$. Moreover, it is also equivalent to the side $\gamma(i)\gamma(i+1)$ being parallel to the side $\gamma(i+n)\gamma(i+n+1)$, or equivalently, $\delta(i+\tfrac{1}{2})=0$, for every $1\leq i\leq 2n$.

The discrete envelope $\E$ of a symmetric polygon reduces to a point, but the converse is not true. There exist non-symmetric half-area polygons whose discrete envelope is just a point (see Figure \ref{fig:HexagonSkewSymmetric}). We shall see below that they occur only for $n$ odd. 

Denote $\lambda(i+\tfrac{1}{2})$ the ratio of the lengths of the central diagonals $\gamma(i+1)\gamma(i+n)$ and 
$\gamma(i)\gamma(i+n+1)$. We say that a half-area polygon is {\it odd-symmetric} if $\lambda(i+\tfrac{1}{2})=c$, $i$ odd, and $\lambda(i+\tfrac{1}{2})=\frac{1}{c}$, $i$ even, for some constant $c$. Of course, an odd-symmetric polygon with $c=1$ is in fact symmetric. Another obvious remark is that, if $c\neq 1$, odd-symmetric polygons exist only for $n$ odd.

\begin{lemma}
$\E(i-\tfrac{1}{2})=\E(i+\tfrac{1}{2})$ if and only if $\lambda(i-\tfrac{1}{2})\cdot\lambda(i+\tfrac{1}{2})=1$.
\end{lemma}

\begin{proof}
Each principal diagonal $l(i)$ and $l(i+1)$
is divided by $\E(i+\tfrac{1}{2})$ into two segments whose ratios are exactly $\lambda(i+\tfrac{1}{2})$ and $\lambda^{-1}(i+\tfrac{1}{2})$. Similarly the principal diagonals $l(i-1)$ and $l(i)$ are divided by $\E(i-\tfrac{1}{2})$ in segments whose ratios are $\lambda(i-\tfrac{1}{2})$ and $\lambda^{-1}(i-\tfrac{1}{2})$. But $\E(i-\tfrac{1}{2})$ and $\E(i+\tfrac{1}{2})$ divide $l(i)$ in segments with the same ratio if and only if they coincide. We conclude that $\E(i-\tfrac{1}{2})=\E(i+\tfrac{1}{2})$ if and only if $\lambda^{-1}(i-\tfrac{1}{2})=\lambda(i+\tfrac{1}{2})$, thus proving the lemma.
\end{proof}

\begin{corollary}
A half-area polygon is odd-symmetric if and only if its discrete envelope reduces to a point. 
\end{corollary}

\begin{figure}[htb]
 \centering
 \includegraphics[width=0.80\linewidth]{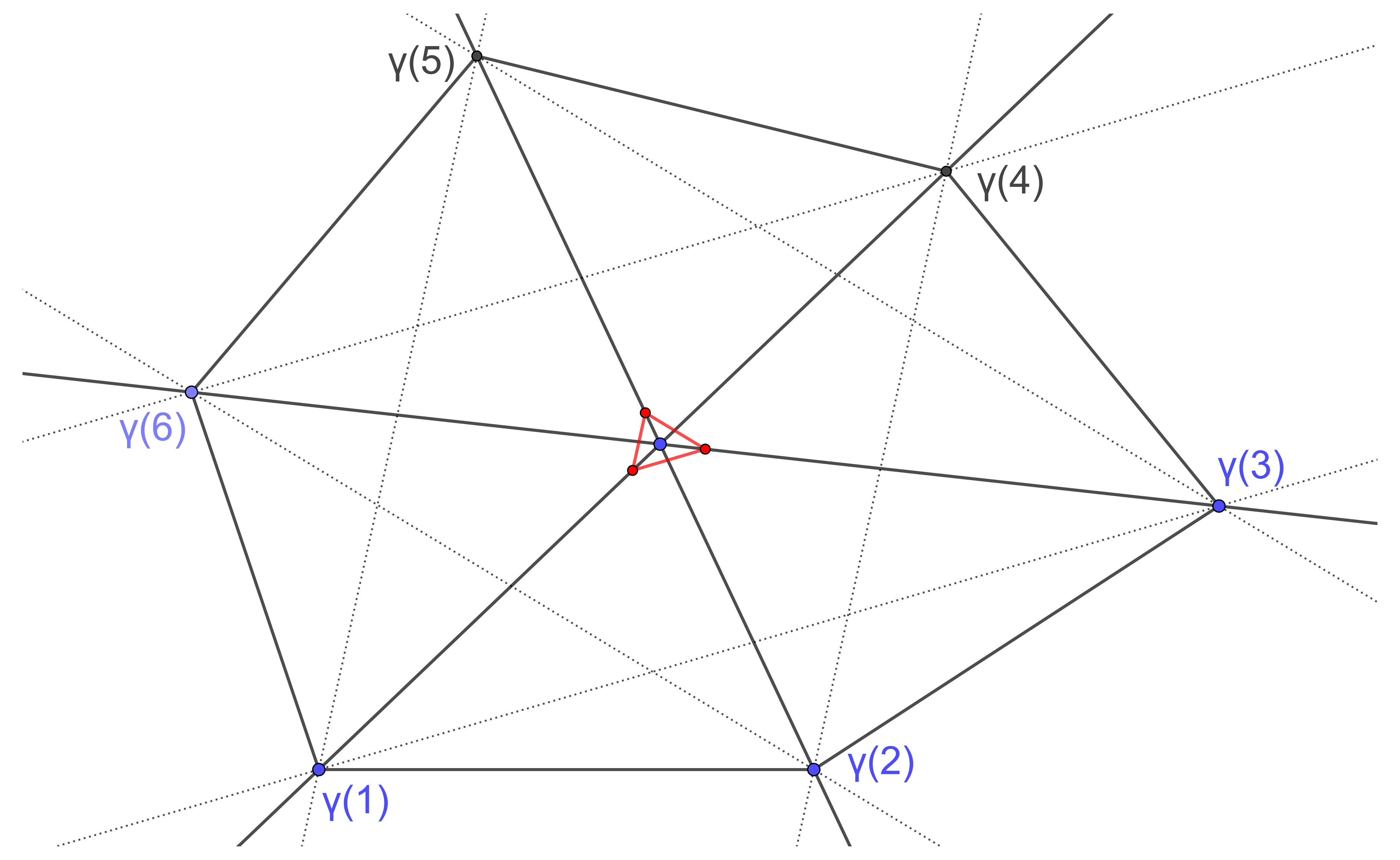}
 \caption{ An odd-symmetric hexagon. Note that the discrete envelope is a point, while the mid-points envelope is a triangle (red).}
\label{fig:HexagonSkewSymmetric}
\end{figure}

\subsection{Half-area vertices and cusps}

In the smooth case, a half-area vertex is a point such that $\alpha=0$. For half-area polygons, we say that $\gamma(i)$ is a {\it half-area vertex} of $\gamma$ if $\alpha_1$ (or $\alpha_2$) changes sign at $i$. Equivalently, we say that $\gamma(i)$ is a half-area vertex if 
\begin{equation*}
\delta(i-\tfrac{1}{2})\cdot\delta(i+\tfrac{1}{2})<0.
\end{equation*}

\begin{lemma}
The midpoint $\M(i)$ is outside the segment $\E(i)$ if and only if $\gamma(i)$ is a half-area vertex of $\gamma$. 
\end{lemma}

\begin{proof}
Geometrically, one can observe that $\M(i)$ belongs to the segment ${\gamma(i)\E(i+\tfrac{1}{2})}$ if and only if
$$
\frac{a^-(i+\tfrac{1}{2})}{2}< \left[\gamma(i)-\E(i+\tfrac{1}{2}), \gamma(i+1)-\E(i+\tfrac{1}{2}) \right].
$$
From Equation \eqref{eq:DefineE}, this is equivalent to 
$$
\frac{a^-(i+\tfrac{1}{2})}{2}<\frac{a^-(i+\frac{1}{2})a^+(i+\frac{1}{2})}{a^2(i+\tfrac{1}{2})}\left[  v(i), v(i+1) \right]
$$
which reduces to $2a^+(i+\frac{1}{2})>a(i+\tfrac{1}{2})$, or equivalently, $\delta(i+\tfrac{1}{2})>0$.

Similarly, one can show that $\M(i)$ belongs to the segment ${\gamma(i)\E(i-\tfrac{1}{2})}$ if and only if
 $\delta(i-\tfrac{1}{2})<0$. We conclude that $\M(i)$ belongs to the segment ${\E(i-\tfrac{1}{2})\E(i+\tfrac{1}{2})}$ if and only if $\delta(i-\tfrac{1}{2})\cdot \delta(i+\tfrac{1}{2})>0$, thus proving the lemma.
\end{proof}

We say that $\M(i)$ is a {\it cusp} of $\M$  if the line $l(i)$ separates the segments $\M(i-\tfrac{1}{2})$ and $\M(i+\tfrac{1}{2})$. Similarly, we say that $\H(i)$ is a {\it cusp} of $\M$  if the line $l(i)$ separates the hyperbolic arcs $\H(i-\tfrac{1}{2})$ and $\H(i+\tfrac{1}{2})$. 

\begin{corollary}
The following statements are equivalent:

\begin{enumerate}
\item 
$\gamma(i)$ is a half-area vertex of $\gamma$,

\item
The midpoint $\M(i)$ is outside the segment $\E(i)$.

\item
$\M(i)$ is a cusp of $\M$.

\item
$\H(i)$ is a cusp of $\H$. 

\end{enumerate}

\end{corollary}

\begin{corollary}
An odd-symmetric polygon has exactly $n$ cusps. 
\end{corollary}

\subsection{Examples with maximal number of cusps}

For $n$ odd, the odd-symmetric polygons are examples of half-area polygons with maximal number $n$ of cusps. 
For $n$ even, there are not half-area polygons with $n$ cusps, since, as we shall see below, the number of cusps of a half-area polygon is always odd. In this section we shall construct $2n$-half-area polygons, $n$ even, with $(n-1)$ cusps.

Starting from a odd-symmetric $2(n-1)$-half-area polygon $\gamma$, we shall describe how to add two more vertices $P$ and $Q$ to $\gamma$ and obtain a $2n$-half-area polygon with $2(n-1)$ cusps (see Figure \ref{fig:MaximalCusps1}). We shall denote by $\gamma\cup\{P,Q\}$ the polygon obtained from $\gamma$ by adding the vertices $P$ and $Q$, $P$ between $\gamma(2n-2)$ and $\gamma(1)$, $Q$ between $\gamma(n-1)$ and $\gamma(n)$.

The following lemma describes some elementary properties of a trapezoid:

\begin{lemma}\label{lemma:Trapezoidal}
Let $ABCD$ be a trapezoid with $AB$ and $CD$ parallel. For any point $P$ in the plane, denote by $Q=Q(P)$ the intersection of the lines parallel to $PD$ through $B$ and parallel to $PA$ through $C$. 
\begin{enumerate}

\item
$Q$ is on the left (resp. at, resp. on the right)
of the line through $AD$ if and only if $P$ is on the right (resp. at, resp. on the left) of the line through $BC$ (Figure \ref{fig:Trapezoids-Diagonals}). 

\item
There is exactly one point $P$ at the segment $BC$
for which $Q$ is at the segment $AD$ and $PQ$ is parallel to $AB$.

\end{enumerate}
\end{lemma}

\begin{figure}[htb]
 \centering
 \includegraphics[width=0.70\linewidth]{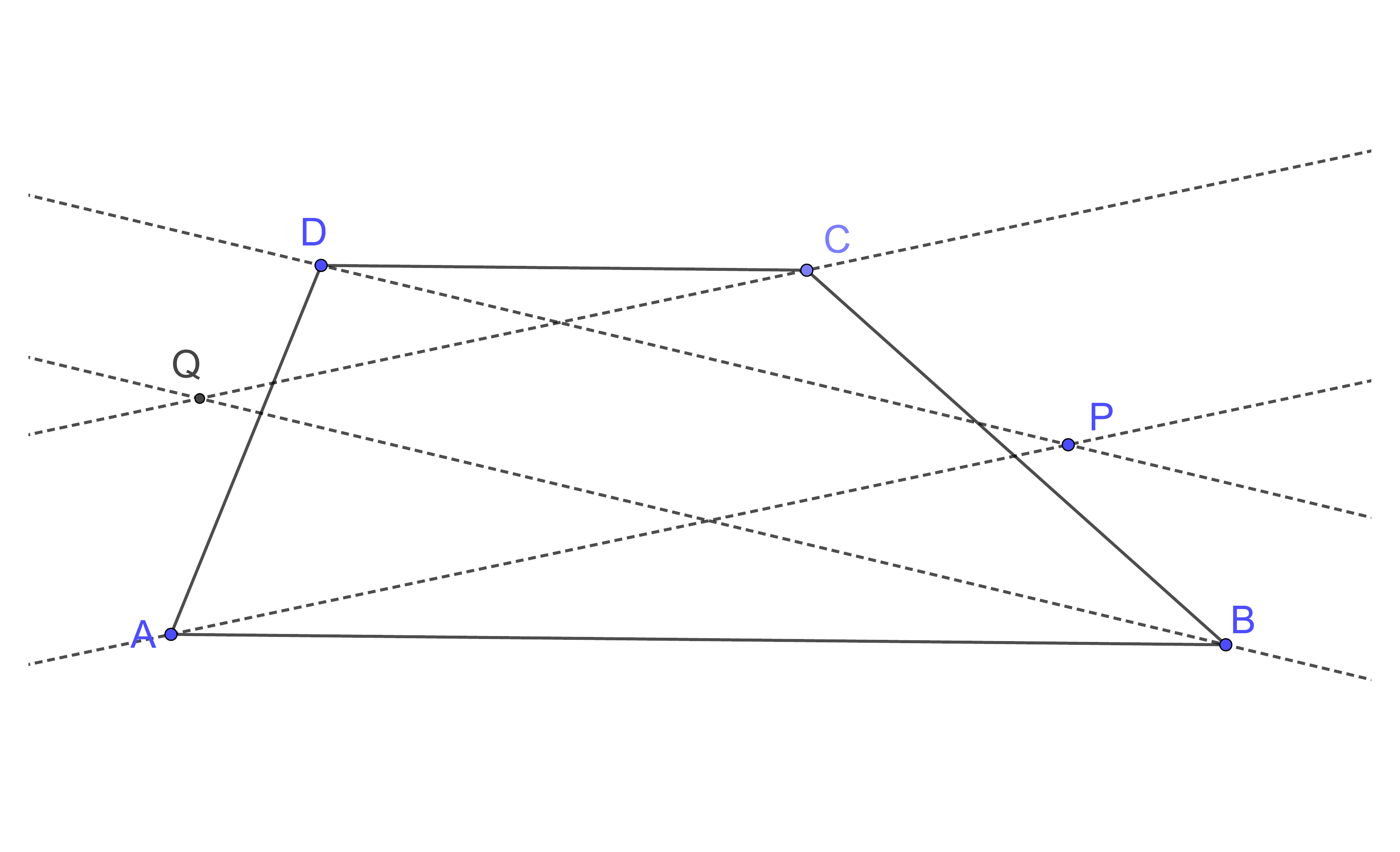}
 \caption{ Illustration of Lemma \ref{lemma:Trapezoidal}.}
\label{fig:Trapezoids-Diagonals}
\end{figure}

\begin{proof}
By an affine change of coordinates, we may assume that $A=(0,0)$, $B=(1,0)$ and $D=(0,1)$. We may also assume that
$C=(c^2,1)$, for some $0<c<1$. Write $P=(x_0,y_0)$. Then straightforward calculations show that $Q=(x_1,y_1)$, where
$$
x_1=1-x_0-(1-c^2)y_0,\ \ y_1=\frac{(1-y_0)(x_0+(1-c^2)y_0)}{x_0}.
$$
But $(x_0,y_0)$ is on the right (resp. at, resp. on the left) of the line through $BC$ if and only if $x_0>1-y_0(1-c^2)$, (resp. $=$, resp. $<$). Thus $(x_0,y_0)$ is on the right (resp. at, resp. on the left) of the line through $BC$ if and only if $x_1<0$, (resp. $=$, resp. $>$), which proves the first item of the lemma. 

For the second item, assume $x_0=1-(1-c^2)y_0$. Then $x_1=0$ and 
$$
y_1=\frac{1-y_0}{1-y_0(1-c^2)}.
$$
Solving the quadratic equation $y_1=y_0$ we obtain 
$y_0=\frac{1\pm c}{1-c^2}$. To obtain $0<y_0<1$, we must choose the minus sign. We conclude that
$$
y_0=\frac{1}{1+c},
$$
which in fact belongs to the interval $(\frac{1}{2},1)$. 
\end{proof}

Now we are in the position of constructing the promised examples. 

\begin{example}\label{ex:MaximalCusps1}
Let $\gamma$ be a $2(n-1)$-half area odd-symmetric polygon and consider the trapezoid 
$\gamma(1)\gamma(n-1)\gamma(n)\gamma(2(n-1))$. By Lemma \ref{lemma:Trapezoidal}(3), there exist unique
$P$ in the segment $\gamma(2(n-1))\gamma(1)$ and $Q$ in the segment $\gamma(n-1)\gamma(n)$
such that $PQ$ is parallel to $\gamma(1)\gamma(n-1)$ and $PQ$ divides the considered trapezoid in equal areas. By the parallelism of central diagonals, one can show that the polygon 
$\gamma\cup\{P,Q\}$ is half-area. Moreover, all mid-points
of principal diagonals area cusps, except for $PQ$ (see Figure \ref{fig:MaximalCusps1}).
\end{example}

\begin{figure}[htb]
 \centering
 \includegraphics[width=0.80\linewidth]{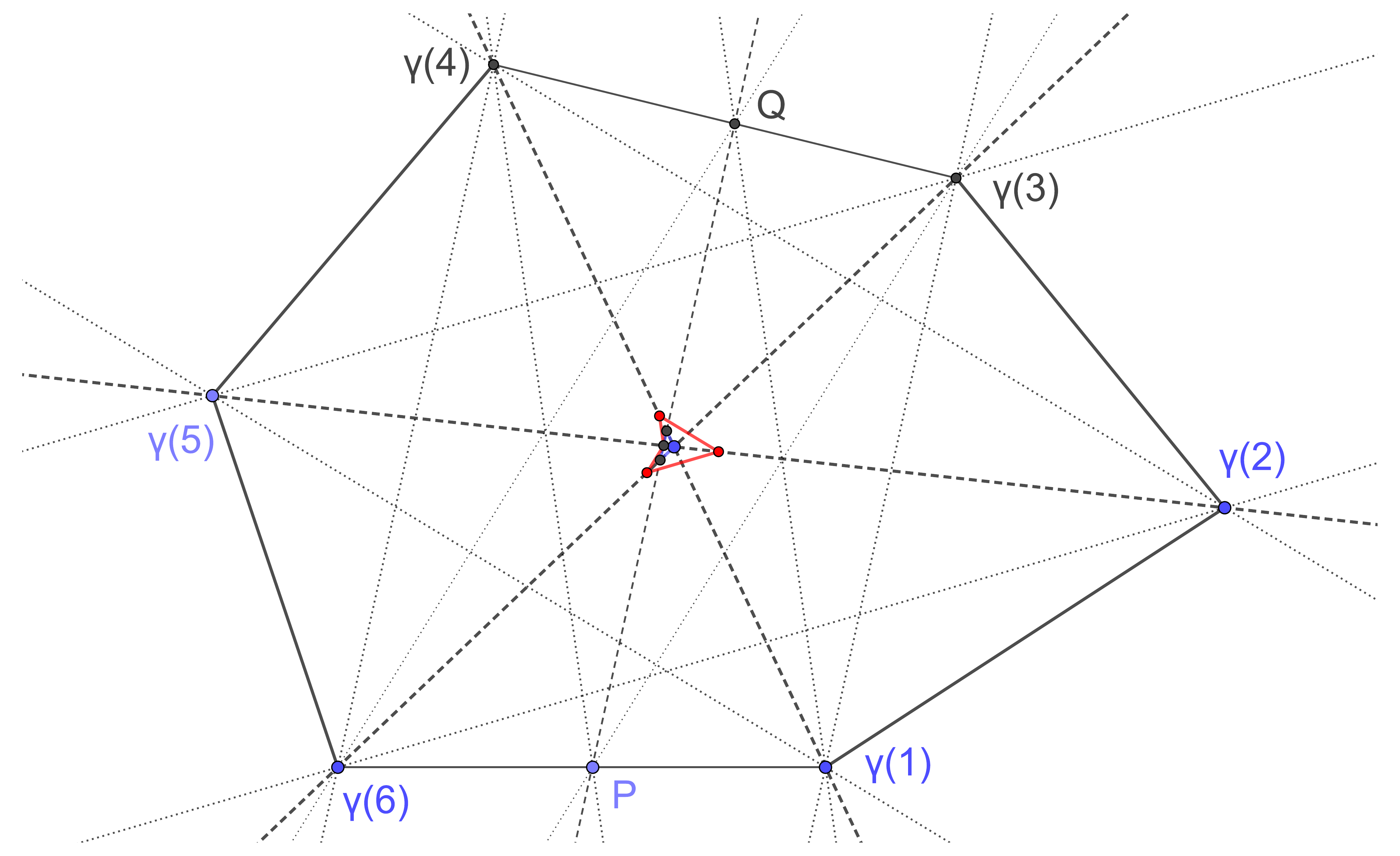}
 \caption{ The octagon with maximal number of cusps of Example \ref{ex:MaximalCusps1}.}
\label{fig:MaximalCusps1}
\end{figure}

The above example has not truly $2n$-sides, since $P$ belongs to $\gamma(1)\gamma(2n-2)$ and $Q$ belongs to $\gamma(n-1)\gamma(n)$. However, we can modify it slightly in order to obtain a convex 
half-area polygon without $3$ collinear vertices. 

\begin{example}\label{ex:MaximalCusps2}
Consider the notation of Example \ref{ex:MaximalCusps1}. We can choose $\bar{P}$ close to $P$ and outside $\gamma$
such that the corresponding $\bar{Q}$ is close to $Q$ and outside $\gamma$ (Lemma \ref{lemma:Trapezoidal}(1)). If 
$\bar{P}$ is sufficiently close to $P$, the polygon $\gamma\cup\{\bar{P},\bar{Q}\}$ is convex without collinear vertices. By the parallelism of central diagonals, the polygon 
$\gamma\cup\{P,Q\}$ is half-area. Moreover, if $\bar{P}$ is sufficiently close to $P$, all midpoints of principal diagonal are cusps, except for $\bar{P}\bar{Q}$
(see Figure \ref{fig:MaximalCusps2}). 
\end{example}

\begin{figure}[htb]
 \centering
 \includegraphics[width=0.80\linewidth]{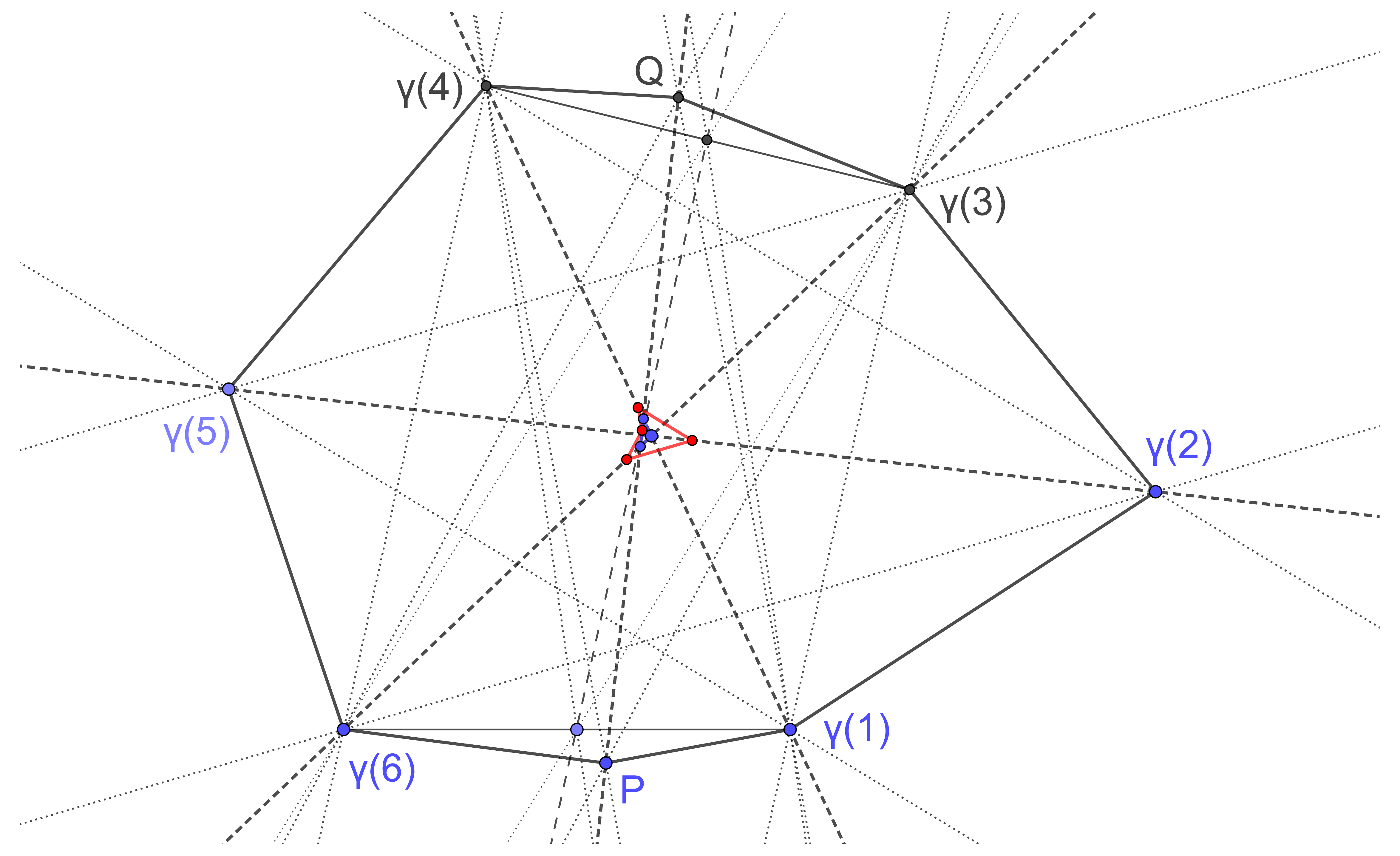}
 \caption{ The octagon with maximal number of cusps of Example \ref{ex:MaximalCusps2}.}
\label{fig:MaximalCusps2}
\end{figure}

\subsection{Three vertices theorem}

In this section we discuss the discrete counterpart of Theorem \ref{thm:ThreeVerticesSmooth}. 
Recall that $\gamma(i)$ is a half-area vertex of $\gamma$ if $\delta(i-\tfrac{1}{2})\cdot \delta(i-\tfrac{1}{2})<0$. Now the condition
$\delta(i+n+\tfrac{1}{2})=-\delta(i+\tfrac{1}{2})$ implies that the number of half-area vertices in the interval $[1,n]$ is odd. We shall verify now that in fact it should be at least three.

Let 
$$
g(i+\tfrac{1}{2})=\frac{\delta(i+\tfrac{1}{2})}{a(i+\tfrac{1}{2})}.
$$
Then $i$ is a half-area vertex if and only if $g(i-\tfrac{1}{2})\cdot g(i-\tfrac{1}{2})<0$. Since $g(i+n+\tfrac{1}{2})=-g(i+\tfrac{1}{2})$, 
we have that
\begin{equation}
\sum_{i=1}^{2n} g(i+\tfrac{1}{2})=0.
\end{equation}

\begin{lemma}\label{lemma:SumZero}
Denote
$$
\bar\gamma(i+\tfrac{1}{2})=\frac{1}{2}\left( \gamma(i)+\gamma(i+1)  \right).
$$
We have that
$$
\sum_{i=1}^{2n} g(i+\tfrac{1}{2})\bar\gamma(i+\tfrac{1}{2})=0.
$$
\end{lemma}

\begin{proof}
We have that
$$
\sum_{i=1}^{2n} g(i+\tfrac{1}{2})\bar\gamma(i+\tfrac{1}{2})=-\sum_{i=1}^{n} g(i+\tfrac{1}{2})\bar{v}(i+\tfrac{1}{2}).
$$
We can write
$$
g(i+\tfrac{1}{2})v(i)-\gamma'(i+\tfrac{1}{2})=\frac{a^+(i+\tfrac{1}{2})}{a(i+\tfrac{1}{2})}v'(i+\tfrac{1}{2}),
$$
$$
g(i+\tfrac{1}{2})v(i+1)-\gamma'(i+\tfrac{1}{2})=\frac{a^-(i+\tfrac{1}{2})}{a(i+\tfrac{1}{2})}v'(i+\tfrac{1}{2}).
$$
Summing we obtain
$$
g(i+\tfrac{1}{2})\bar{v}(i+\tfrac{1}{2})-\gamma'(i+\tfrac{1}{2})=\frac{1}{2}v'(i+\tfrac{1}{2}).
$$
Thus 
$$
g(i+\tfrac{1}{2})\bar{v}(i+\tfrac{1}{2})=\frac{1}{2}\left(\gamma'(i+\tfrac{1}{2})+\gamma'(i+n+\tfrac{1}{2})\right).
$$
$$
\sum_{i=1}^{n} g(i+\tfrac{1}{2})\bar{v}(i+\tfrac{1}{2})=\frac{1}{2}\left(\gamma(n+1)-\gamma(1)+\gamma(1)-\gamma(n+1)  \right)=0,
$$
thus proving the lemma.
\end{proof}

\begin{thm}
The number of half-area vertices in the interval $[0,T_0]$ is odd and at least $3$.
\end{thm}

\begin{proof}
Assume by contradiction that there is only one half-area vertex $i_0$ in the interval $[1,n]$. Write the mid-area line passing through $\gamma(i_0)=(x_0,y_0)$ as $A(x-x_0)+B(y-y_0)=0$. We may assume that $\delta>0$ in the region $A(x-x_0)+B(y-y_0)>0$ (or else change the signs of $A$ and $B$). Then 
$$
\sum_{i_0}^{i_0+n}\delta (Ag_1+Bg_2)>0, \ \ \sum_{i_0+n}^{i_0+2n}\alpha (Ag_1+Bg_2)>0,
$$
where $g=(g_1,g_2)$. But these contradicts Lemma \ref{lemma:SumZero}.
\end{proof}

\section{Inverse Construction}

Consider the polygon $\bar{\gamma}$ defined by
\begin{equation}\label{eq:involute}
\bar\gamma(i)= (1-c(i)) \gamma(i)+ c(i) \gamma(i+n),
\end{equation}
where $c(i)$ depends on $i$.

\begin{lemma}
We have that:
\begin{enumerate}
\item
$$
\bar{v}(i)=(1-2c(i))v(i).
$$
\item
$$
\bar\gamma'(i+\tfrac{1}{2})=(1-c(i))\gamma'(i+\tfrac{1}{2})+c(i)\gamma'(i+n+\tfrac{1}{2})+c'(i+\tfrac{1}{2})v(i+1),
$$
or equivalently,
$$
\bar\gamma'(i+\tfrac{1}{2})=(1-c(i+1))\gamma'(i+\tfrac{1}{2})+c(i+1)\gamma'(i+n+\tfrac{1}{2})+c'(i+\tfrac{1}{2})v(i).
$$
\item
$$
\bar{a}^-(i+\tfrac{1}{2})=(1-2c(i+1))\left( a^-(i+\tfrac{1}{2})-c(i)a(i+\tfrac{1}{2})  \right)
$$
and
$$
\bar{a}^+(i+\tfrac{1}{2})=(1-2c(i))\left(a^+(i+\tfrac{1}{2})-c(i+1)a(i+\tfrac{1}{2})\right).
$$
\end{enumerate}
\end{lemma}

\begin{proof}
Straightforward calculations.
\end{proof}

The following proposition is a discrete counterpart of Proposition \ref{prop:RecoveringSmooth}:

\begin{Proposition}
Assume that no opposite sides are parallel. 
If $\E(\bar\gamma)=\E(\gamma)$ and $\M(\bar\gamma)=\M(\gamma)$, then $\bar\gamma$ has the form
\eqref{eq:involute} with $c$ constant.  Conversely, any polygon $\bar\gamma$ of the form \eqref{eq:involute} with $c$ constant
has the same $\E$ and $\M$ envelopes as $\gamma$.
\end{Proposition}

\begin{proof}
Assuming $\E(\bar\gamma)=\E(\gamma)$ and $\M(\bar\gamma)=\M(\gamma)$, the line $l(i)$ and the midpoint $\M(i)$ are the same for $\gamma$ and $\bar\gamma$, which implies that $\bar\gamma$ has the form \eqref{eq:involute}. Now we have to check that $l(i)$ is also a mid-area line for $\bar\gamma$, which is equivalent to 
\begin{equation}
\bar{a}^+(i+\tfrac{1}{2})=\bar{a}^-(i+n+\tfrac{1}{2}).
\end{equation}
for any $1\leq i\leq 2n$. This condition is equivalent to
$$
\left( 2a^+(i+\tfrac{1}{2})-a(i+\tfrac{1}{2}) \right)c'(i+\tfrac{1}{2})\neq 0.
$$
Under the hypothesis that the opposite sides ${\gamma(i)\gamma(i+1)}$ and ${\gamma(i)\gamma(i+1)}$
are not parallel, this condition is equivalent to $c(i)=c(i+1)$. Since this holds for any $i$, we conclude that $c$ is constant.
For the converse, just follow the steps backwards. 
\end{proof}

Thus the polygons with the same sets $\E$ and $\M$ as $\gamma$ form a one parameter family given by
\begin{equation*}
\gamma_c(i)= (1-c) \gamma(i)+ c \gamma(i+n),
\end{equation*}
where $c\in\mathbb{R}$. However, we should point out that $\H(\gamma_c)\neq \H(\gamma)$. In fact, the hyperbola containing the arc $\H(\gamma)(i+\tfrac{1}{2})$ is 
asymptotic to the sides $\gamma(i+\tfrac{1}{2})$ and $\gamma(i+n+\tfrac{1}{2})$ of $\gamma$, while the hyperbola containing the arc $\H(\gamma_c)(i+\tfrac{1}{2})$ is 
asymptotic to the correponding sides of $\gamma_c$. 

\begin{figure}[ht]
	\centering
	\includegraphics[width=0.80\linewidth]{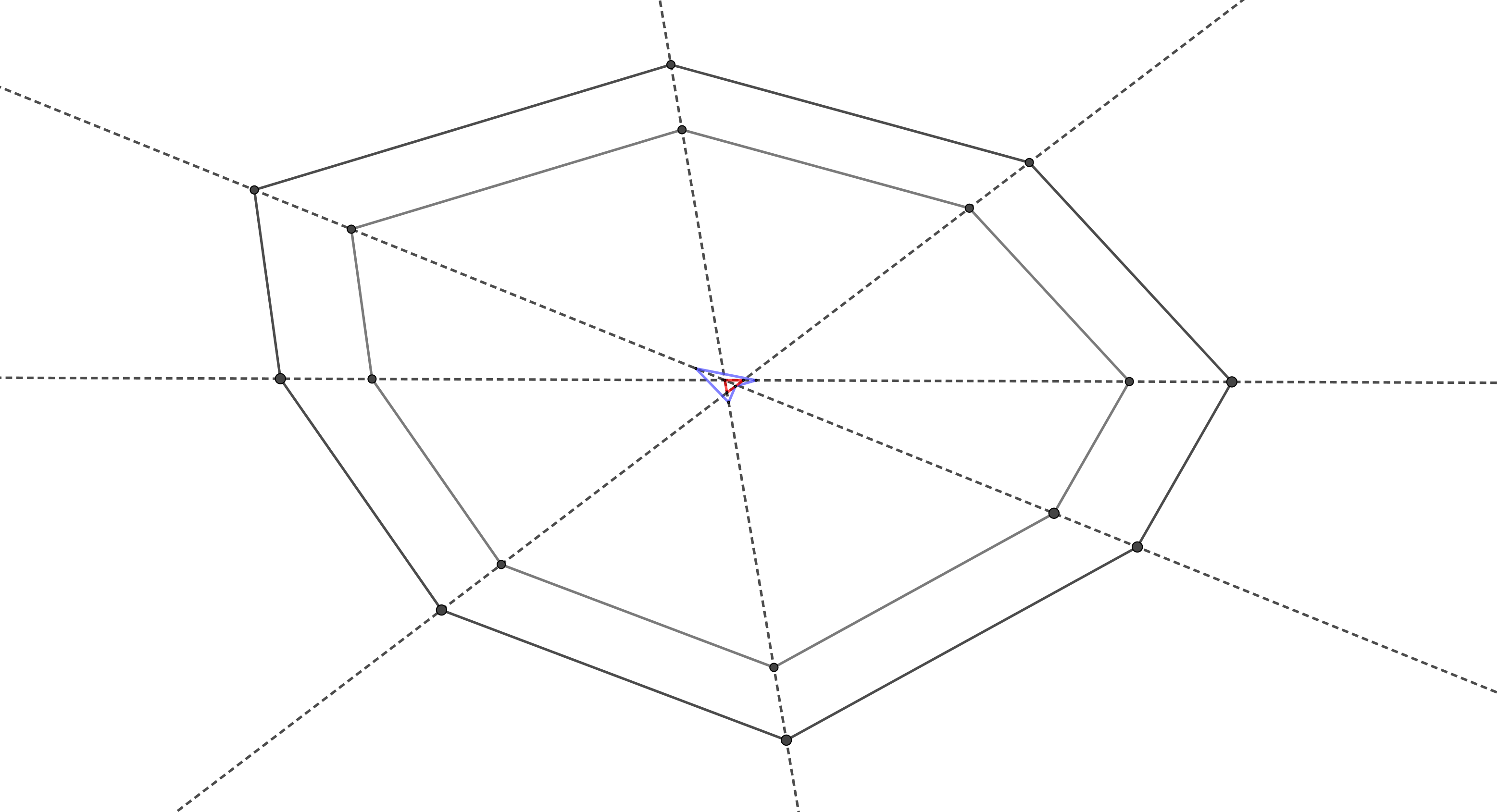}
	\caption{The octagon $\gamma$ (outer) together with the octagon $\gamma_c$ (inner), for $c=0.1$. The envelopes $\E$ and $\M$ are the same for both, but $\H$ is different.}
	\label{fig:test1}
\end{figure}

In fact, we have just proved that if $\H(\bar\gamma)=\H(\gamma)$ and $\gamma$ has no opposite parallel sides, 
then $\bar\gamma=\gamma$, thus recovering the result of \cite{Noah}.






\end{document}